\documentclass[12pt]{article}
\usepackage{theorem,amsmath,amssymb,latexsym}
\usepackage[all]{xy}

\theoremstyle{plain}
\theoremheaderfont{\normalfont\bfseries}
\theorembodyfont{\itshape}

\newtheorem{theorem}{Theorem}[section]
\newtheorem{proposition}[theorem]{Proposition}
\newtheorem{lemma}[theorem]{Lemma}

\newtheorem{conjecture}{Conjecture}[section]

\theorembodyfont{\upshape}
{\theoremstyle{break}

}

\theorembodyfont{\upshape}
\newtheorem{remark}[theorem]{Remark}

\theorembodyfont{\upshape}

\theorembodyfont{\upshape}

\theorembodyfont{\upshape}

\newtheorem{problem}[theorem]{Problem}

\theorembodyfont{\itshape}

\theorembodyfont{\upshape}

\theorembodyfont{\upshape}

\theorembodyfont{\upshape}

\theorembodyfont{\upshape}

\theorembodyfont{\upshape}

\theorembodyfont{\upshape}
\newtheorem{proof}{Proof}

\theorembodyfont{\upshape}
\newtheorem{sketch}{Sketch of Proof}

\newtheorem{acknowledgement}{Acknowledgement}

\newcommand{\qed}{\nopagebreak\par
\hspace*{\fill}$\square$\par\vskip2mm}
\newcommand{\myemail}[1]{\indent \emph{E-mail:} {\tt #1}}
\newcommand{\myaddress}[1]{\indent {\sc #1}\par}

\usepackage{verbatim}

\allowdisplaybreaks
\newcommand{\Tr}{\mathrm{Tr}}

\renewcommand{\Im}{\mathrm{Im}}

\newcommand{\Sp}{\mathrm{Sp}}
\renewcommand{\H}{\mathbb H}
\newcommand{\T}[1]{{}^t{{#1}}}

\newcommand{\A}{{\mathbb A}}
\renewcommand{\a}{{\mathfrak a}}
\newcommand{\Q}{{\mathbb Q}}
\newcommand{\Z}{{\mathbb Z}}
\newcommand{\R}{{\mathbb R}}
\newcommand{\C}{{\mathbb C}}
\newcommand{\bs}{\backslash}

\newcommand{\GL}{{\rm GL}}
\newcommand{\PGL}{{\rm PGL}}
\newcommand{\SL}{{\rm SL}}

\newcommand{\GSp}{{\rm GSp}}
\newcommand{\PGSp}{{\rm PGSp}}

\renewcommand{\S}{{\mathcal S}}

\newcommand{\disc}{{\rm disc}}

\DeclareMathOperator{\Cl}{Cl}
\newcommand{\mat}[4]{{\setlength{\arraycolsep}{0.5mm}
\left(\begin{array}{cc}#1&#2\\#3&#4\end{array}\right)}}

\newcommand{\forget}[1]{}

\begin{document}

\title{Determination of modular forms by fundamental Fourier coefficients}

\author{Abhishek Saha}

\begin{abstract}
It is an interesting question when a natural subset of the
Fourier coefficients is sufficient to uniquely determine a modular
form.  This article deals with this question for two kinds of modular forms: a) classical modular forms of half-integral weight, and b) Siegel
modular forms of genus 2 and integral weight. These two apparently different scenarios
turn out to be closely related. Our results have several applications to automorphic
$L$-functions and Bessel models. \end{abstract}

 \maketitle

\section{Introduction}
Let $V$  be a fixed set consisting of some ``modular forms".\footnote{We are being deliberately vague here in order to include a wide variety of objects that generalize or are similar to classical modular forms.} Let $\mathcal{S}$ be an indexing set for the ``Fourier coefficients" of elements of $V$. This means that for all $\Phi \in V$, we have an expansion\footnote{The model to keep in mind here is the well-known Fourier expansion of classical modular forms.}
$$\Phi(z) = \sum_{n\in \mathcal{S}} \Phi_n(z).$$
Let $\mathcal{D}$ be an \emph{interesting} subset of $\mathcal{S}$. We are interested in situations where the following implication is true for all $\Phi \in V$: $\Phi_n = 0 \ \forall \ n \in \mathcal{D}\Rightarrow  \Phi = 0$;
 equivalently: $ \Phi \neq 0  \Rightarrow \text{ there exists } n \in \mathcal{D} \text{ such that } \Phi_n \neq 0.$ Another way of phrasing this problem is:
\emph{When does an interesting subset of Fourier coefficients determine a modular form}?

In this article, we will consider a special case of this question for the following two types of modular forms:

\begin{enumerate}
\item Classical modular forms of half-integral weight (automorphic forms on $\widetilde{\SL_2}$)

\item Siegel modular forms of degree 2 and trivial central character (automorphic forms on $\PGSp_4$)

\end{enumerate}

For both types, the interesting subset $\mathcal{D}$ will be related to $\S$ in a similar manner as  the set of squarefree integers is related to the set of all integers. Moreover, as we will see, the proof of our main result for Siegel modular forms hinges crucially on the corresponding result for classical modular forms of half-integral weight.

This article in based on the talk the author gave at the \emph{International Colloquium on
Automorphic Representations and L-Functions, 2012,}
held at the Tata Institute of
Fundamental Research. Some of the results of this article were obtained jointly with Ralf Schmidt. Most of the results are contained either in the author's paper~\cite{sahafund} or in his joint paper with Schmidt~\cite{sahaschmidt}. The presentation here will therefore be less formal and of a more descriptive nature; our goal is to tie several related results into a coherent framework and emphasize the underlying motivations and ideas. The reader is invited to consult~\cite{sahafund} and~\cite{sahaschmidt}  for the technical details of the proofs that are sketched or omitted from this article.

We briefly summarize the structure of this article. In Section~\ref{halfintsection} we deal with the case of half-integral weight modular forms. In Section~\ref{s:siegel}, we deal with the case of Siegel cusp forms of degree 2. The proofs of Section~\ref{s:siegel} need the results of Section~\ref{halfintsection} in an essential manner. Finally, in Section~\ref{s:app}, we give some important applications of our results to global Bessel models for $\GSp_4$ and simultaneous non-vanishing of dihedral twists of modular $L$-functions.

\section{The case of half-integral weight modular forms }\label{halfintsection}

\subsection{Notations and preliminaries}\label{halfintsec}
The group $\SL_2(\R)$ acts on the upper half-plane $\H$ via $$\gamma z := \frac{az+b}{cz+d}$$ for $\gamma = \mat{a}{b}{c}{d} \in \SL_2(\R)$ and $z=x+iy \in \H $. For a positive integer $N$, let $\Gamma_0(N)$ denote the congruence subgroup consisting of matrices $\mat{a}{b}{c}{d}$ in $\SL_2(\Z)$ such that $N$ divides $c$. For a complex number $z$, let $e(z)$ denote $e^{2\pi i z}$.

Let $\theta(z) = \sum_{n = -\infty}^\infty e(n^2 z)$ be the standard theta function on $\H$. If $A = \mat{a}{b}{c}{d} \in \Gamma_0(4)$, we have $\theta(Az) = j(A, z)\theta(z),$ where $j(A, z)$ is the so called $\theta$-multiplier. For an explicit formula for $j(A, z)$, see~\cite{Shimura3} or~\cite{Serre-Stark}.

Let $k$ be an integer. For a positive integer $N$ divisible by $4$, let $S_{k+\frac{1}{2}}(N)$ denote the space of holomorphic cusp forms of weight $k+\frac{1}{2}$ for the group $\Gamma_0(N)$. Precisely, a function $f : \H \rightarrow \C$ belongs to $S_{k+\frac{1}{2}}(N)$ if
\begin{enumerate}

\item $f(Az) = j(A, z)^{2k +1} f(z)$ for every $A = \mat{a}{b}{c}{d} \in \Gamma_0(N)$,
\item $f$ is holomorphic,
\item $f$ vanishes at the cusps.

\end{enumerate}

Any $f \in  S_{k+\frac{1}{2}}(N)$ has the Fourier expansion

$$f(z) = \sum_{n > 0} a(f, n)e(nz).$$ We let $\tilde{a}(f,n)$ denote the ``normalized" Fourier coefficients, defined by
 $$\tilde{a}(f,n) = a(f,n)n^{ \frac14-\frac k2}.$$ The Kohnen plus-space $S^+_{k+\frac12}(N)$ is defined to be the subspace of $S_{k+\frac{1}{2}}(N)$ consisting of forms $f$ for which $a(f,n) = 0$ whenever $n \equiv (-1)^{k+1}$ or $2 \mod{4}$. Kohnen~\cite{kohnew} developed a theory of \emph{newforms} for the space $S_{k+\frac{1}{2}}^+(N)$ in the case that $N/4 $ is odd and squarefree. He also proved~\cite{kohnwalds} a precise version of Waldspurger's theorem in this setting. If $f \in S_{k+\frac{1}{2}}^+(N)$ is a newform  and $(-1)^k d$ is a \emph{fundamental discriminant},\footnote{Recall that an integer $n$ is a fundamental discriminant if \emph{either} $n$ is a squarefree integer congruent to 1 modulo 4 \emph{or} $n = 4m$ where $m$ is a squarefree integer congruent to 2 or 3 modulo 4.} then Kohnen's formula implies that $|\tilde{a}(f,d)|^2$ is essentially equal to $L(\frac12, \pi_f\times \chi_{(-1)^kd})$; here $\pi_f$ is the automorphic representation of $\PGL_2$ attached to $f$ by the Shimura correspondence and $\chi_{(-1)^kd}$ is the quadratic Dirichlet character associated via class field theory to the field $\Q(\sqrt{(-1)^k d})$.

\subsection{The problem}

Let $\mathcal{S}$ equal the set of integers $n$ such that $(-1)^k n$ is a discriminant, i.e., $n \equiv 0$ or 1 mod $4$. If $f \in S^+_{k+\frac12}(N)$, the Fourier coefficients $a(f, n)$ are indexed by $n \in \S$. A natural subset of $\S$ is the set $\mathcal{D}$ consisting of integers $n$ such that $(-1)^k n$ is a fundamental discriminant. It is clear from the Kohnen--Waldspurger formula that if $d \in \mathcal{D}$ and $f \in S^+_{k+\frac12}(N)$ is a newform (so in particular, a Hecke eigenform at all places), then $a(f, d)$ has deep arithmetic significance. So, it is a natural question to ask whether elements of $S_{k+ \frac12}^+ (N)$ are determined by the Fourier coefficients $a(f, d)$ with $d \in \mathcal{D}$.

In the language of the introduction, we want to prove that given any non-zero modular form $\Phi$ in $V$, $\text{ there exists } n \in \mathcal{D} \text{ such that } a(f,n) \neq 0.$ Here $V$ is a suitable subset of $S^+_{k+\frac12}(N)$. The simplest version of the problem is when $V$ is a finite Hecke basis consisting of newforms. Here, we may assume that $N=4M$ where $M$ is odd and squarefree\footnote{These are the assumptions under which Kohnen developed his newform theory.}. In this case, the problem was solved by Kohnen. The proof of Kohnen's result, stated immediately below, is straightforward and uses only some combinatorics with Hecke operators at the prime 2.

\begin{theorem}[Kohnen, p. 70 of ~\cite{kohnew}]\label{t:kohsimple}
Suppose that $M$ is odd and squarefree and $0 \neq f \in S^+_{k+\frac12}(4M)$ is a newform. Then there exists a $d$ such that $(-1)^k d$ is a fundamental discriminant and $a(f,d) \neq 0$.

\end{theorem}

The next simplest version of the question is obtained by taking $V$ to be set of all elements of the form $a_1g_1 + a_2 g_2$ where $g_1$ and $g_2$ lie in a fixed Hecke basis consisting of newforms in $S^+_{k+\frac12}(N)$. This question was raised explicitly by Kohnen~\cite{kohn92} in 1992 and solved by Luo--Ramakrishnan~\cite{luoram} in 1997.

\begin{theorem}[Luo-Ramakrishnan~\cite{luoram}] Suppose that $M$ is odd and squarefree. Let $g_1$ and $g_2$ be multiples of newforms  in $S^+_{k+\frac12}(4M)$. Assume further that $a(g_1, d) = a(g_2, d)$ for all $d$ such that $(-1)^k d$ is a fundamental discriminant. Then $g_1 = g_2$.

\end{theorem}

The above theorem clearly implies the result of Kohnen stated earlier (simply take $g_2 =0$). The proof of the theorem is quite involved and uses the Kohnen--Waldspurger formula in a crucial way. Indeed, the squares of Fourier coefficients of half-integral weight newforms are related to twisted $L$-values of integral weight  newforms; Luo--Ramakrishnan were able to prove that integral weight newforms are uniquely determined by the central $L$-values of their twists with quadratic characters.

Finally, one may consider the hardest (and most general) version of the question, where we take the set $V$ to be the entire set $S^+_{k+\frac12}(4M)$. Thus the problem becomes:

\begin{problem}\label{problem1}Let $f \in S^+_{k+\frac12}(4M)$, $M$ squarefree, $f \neq 0$. Does there exist $d$ such that  $(-1)^k d$ is a fundamental discriminant and $a(f,d) \neq 0$?

\end{problem}

Note that a solution to the above problem will automatically imply the result of Luo--Ramakrishnan (take $f = g_1 - g_2$). However, because $f$ is no longer assumed to be a Hecke eigenform or a difference of two eigenforms, it appears impossible to reduce the problem to one about central $L$-values via Waldspurger's formula. This is in sharp contrast to the results stated earlier in this subsection.

\subsection{The main result for half-integral forms}

 The next theorem gives an affirmative answer, in a strong quantitative form, to the problem posed in the previous subsection, whenever $k \ge 2$. Note that the statement of the result does not involve the Kohnen plus-space $S_{k+\frac12}^+(4M)$ at all, but only the larger space  $S_{k+\frac12}(4M)$.

\begin{theorem}\label{t:mainhalf}Let $k \ge 2$ and let $M$ be a squarefree integer. Suppose that $f \in S_{k+\frac12}(4M)$ is non-zero. Then, one has the lower bound

$$ \# \{0< d < X : d \text{ squarefree, } a(f,d) \neq 0 \} \gg_{f, \delta} X^\delta,$$ where $\delta > 0$ is an absolute constant (any value of $\delta < 5/8$ is admissible).
In particular, there are infinitely many squarefree integers $d$ such that $a(f,d) \neq 0$.
\end{theorem}
The above theorem is not true if $k <2$; counter-examples can be obtained by considering suitable theta-series. Note also that if $k \ge 2$, it gives a positive answer to Problem~\ref{problem1}. Indeed, if $f \in  S_{k+\frac12}^+(4M)$, then $a(f,n)$ is supported on the integers $n$ for which $(-1)^kn$ is a discriminant. If such a $n$ is squarefree, it is automatically a fundamental discriminant.

Theorem~\ref{t:mainhalf} is a special case of a more general result (involving not necessarily squarefree levels and possibly nontrivial nebentypus) that was proven in~\cite{sahafund} and further generalized in~\cite{sahaschmidt}. So, we will only briefly sketch the proof here; the reader is invited to consult~\cite{sahafund} for further details.

\begin{sketch}
For any positive integer $M$, define $$S(M,X;f) := \sum_{\substack{d\text{ squarefree } \\ (d, M) = 1}} |\tilde{a}(f, d)|^2 e^{-d/X}.$$
Suppose that we can show there exists $M$ such that $S(M,X;f) \gg_f X$. The result then follows immediately from any bound of the form $$|\tilde{a}(f, d)|^2 \ll_{f,\delta} d^{1 - \delta}.$$ The first non-trivial bound of this form is due to Iwaniec~\cite{iwanfourier} who showed that one can take any $\delta < 4/7$. The best current bound is due to Bykovski{\u\i}~\cite{byko} who improved this to $5/8$; see also Blomer--Harcos~\cite{blomer-harcos08}. The (unproven) Lindel\"of hypothesis for twisted $L$-functions of $f$ implies that one take any $\delta < 1$.

So the key point is to prove that $S(M,X;f) \gg_f X$ for some $M$. For this, we re-write $S(M, X;f)$ using a ``squarefree sieve", as follows.

$$S(M,X;f)  = \sum_{\substack{\ r \text{ squarefree } \\ (r, M) = 1}} \mu(r) \sum_{\substack{n >0 \\ (n, M) = 1}} |\tilde{a}(f, nr^2)|^2e^{-r^2n/X} ,$$ where $\mu(n)$ denotes the Mobius function. We can bound the leading term (corresponding to $r=1$) from below by using a theorem of Duke and Iwaniec~\cite{dukeiwan}. The other terms can be bounded from above by using the theory of Hecke operators and the Deligne bound in the Hecke eigenvalues (see~\cite{sahafund} for details). Putting things together, one sees that for large enough $M$, we have $S(M,X;f) \gg_f X$. As noted earlier, this completes the proof.
\qed
\end{sketch}

\section{The case of Siegel cusp forms of degree 2}\label{s:siegel}

\subsection{Notations and preliminaries}

For any commutative ring $R$ and positive integer $n$, let $M_n(R)$
  denote the ring of $n$ by $n$ matrices with entries in $R$ and $\GL_n(R)$ denote the group of invertible matrices.  If
  $A\in M_n(R)$, we let $\T{A}$ denote its transpose. We say that a symmetric matrix in $M_n(\Z)$ is semi-integral if
  it has integral diagonal entries and half-integral off-diagonal
  ones.  Denote by $J$ the $4$ by $4$ matrix given by
$$
J =
\begin{pmatrix}
0 & I_2\\
-I_2 & 0\\
\end{pmatrix}.
$$ where $I_2$ is the identity matrix of size 2.

Define the algebraic group
  $\Sp_4$ over $\Z$ by
$$
\Sp_4(R) = \{g \in \GL_4(R) \; | \; \T{g}Jg =
  J\}
$$
for any commutative ring $R$.

The Siegel upper-half space of degree 2 is defined by
$$
\H_2 = \{ Z \in M_2(\C)\;|\;Z =\T{Z},\ \Im(Z)
  \text{ is positive definite}\}.
$$
The group $\Sp_4(\R)$ acts on $\H_2$ via
$$
 gZ := (AZ+B)(CZ+D)^{-1}\qquad\text{for }
 g=\begin{pmatrix} A&B\\ C&D \end{pmatrix} \in \Sp_4(\R),\;Z\in \H_2.
$$
We let $J(g,Z) = CZ + D$.
For any positive integer $N$,
define
\begin{equation}\label{Gamma0defeq}
\Gamma_0^{(2)}(N) := \left\{\begin{pmatrix}A&B\\ C&D \end{pmatrix} \in \Sp_4(\Z)\;|\;C \equiv 0 \pmod{N}\right\}.
\end{equation}
Let $S_k^{(2)}(N)$ denote the space of holomorphic functions $F$ on
$\H_2$ which satisfy the relation
\begin{equation}\label{siegeldefiningrel}
F(\gamma Z) = \det(J(\gamma,Z))^k F(Z)
\end{equation}
for $\gamma \in \Gamma_0^{(2)}(N)$, $Z \in \H_2$, and vanish at all the
cusps. Elements of $S_k^{(2)}(N)$ are often referred to as Siegel cusp forms of degree (genus) 2, weight $k$ and level $N$. The Siegel cusp forms of degree 2 can be viewed as the simplest generalization of the classical cusp forms, and they arise naturally in number theory and representation theory. They also have applications to coding theory and conformal field theory.

  Any $F\in S_k^{(2)}(N)$ has a Fourier expansion $$F(Z)
=\sum_{T > 0} a(F, T) e(\Tr(TZ)),$$ where $T$ runs through all symmetric, semi-integral, positive-definite matrices of
size two, or equivalently, all positive, integral, binary quadratic forms. In fact, because $\begin{pmatrix}A&0\\0&\T{A}^{-1} \end{pmatrix} \in \Gamma_0^{(2)}(N)$ for all $A \in \SL_2(\Z)$, we have, using~\eqref{siegeldefiningrel},  that \begin{equation}\label{siegelinv}
a(F, \T{A}TA) = a(F,T)
\end{equation}
for all $A \in \SL_2(\Z)$, thus showing that $a(F,T)$ only depends on the $\SL_2(\Z)$-equivalence class of $T$, i.e., only on the proper equivalence class of the associated binary quadratic form.

We denote by $ S_k^{(2), \text{O}}(N)$ the linear subspace of  $ S_k^{(2)}(N)$ spanned by the set $$\{F(dZ) : F \in S_k^{(2)}(M), \ dM | N, \ M \neq N \}.$$
Note that if $N=1$, then $ S_k^{(2), \text{O}}(N) = \{0 \}$.

\subsection{The problem}

If $F \in S_k^{(2)}(N)$, the Fourier coefficients $a(F, S)$ are indexed by $S \in \S$, where $\mathcal{S}$ equals the set of matrices of the form
$$
 S=\mat{a}{b/2}{b/2}{c},\qquad a,b,c\in\Z, \qquad a>0, \qquad \disc(S) := b^2 - 4ac < 0.
$$

If $\gcd(a,b,c)=1$, then $S$ is called \emph{primitive}. If $\disc(S)$ is a fundamental discriminant, then $S$ is called \emph{fundamental}. Observe that if $S$ is fundamental, then it is automatically primitive. Observe also that the following conditions are equivalent: a) $\disc(S)$ is squarefree, b) $S$ is fundamental with $\disc(S)$ odd.

One possible choice of the interesting subset $\mathcal{D}$ of $\S$ consists of all the primitive matrices. In this context, the problem stated in the introduction was solved in the special case $N=1$ by Zagier \cite[p.\ 387]{zagier81}. This result was generalized by Yamana \cite{yamana09} to cusp forms with level and of higher degree. Using Yamana's result, Ibukiyama and Katsurada proved the following theorem.\footnote{The Theorem of Ibukiyama--Katsurada is actually for Siegel cusp forms of any degree but we state it only for degree 2 here for convenience.}

\begin{theorem}[Ibukiyama--Katsurada~\cite{ibukat11}]
\label{t:ibukat} Let $F \in S_k^{(2)}(N)$ belong to the orthogonal complement of $ S_k^{(2), \text{O}}(N)$ and suppose that $a(F, S) = 0$ for all the primitive matrices $S$. Then $F=0$.

\end{theorem}

In the language of the introduction, Theorem~\ref{t:ibukat} addresses the case when $V$ equals the orthogonal complement  of $ S_k^{(2), \text{O}}(N)$ and $\mathcal{D}$  equals the subset of $\S$ consisting of the primitive matrices. Note that it is not possible to extend Theorem~\ref{t:ibukat} so that $V$ equals the full set $S_k^{(2)}(N)$. This is because forms such as $F(NZ)$ with $F \in S_k^{(2)}(1)$ are only supported on non-primitive coefficients.

From the point of view of  representation theory, it is the smaller set of fundamental Fourier coefficients that are more interesting  than the primitive ones. In the language of the introduction, let $\S$ be as before and let $\mathcal{D}$ consist of the fundamental matrices. Then, the simplest question in this direction is obtained by taking $V$ to be a finite set consisting of eigenforms at all places.

\begin{problem}\label{siegelsimp}Let $F \in S_k^{(2)}(N)$ be a non-zero element in the orthogonal complement of $ S_k^{(2), \text{O}}(N)$ and an eigenfunction for the local Hecke algebras at all primes. Show that there exists a fundamental matrix $S$ such that $a(F,S) \neq 0$.
\end{problem}

The above problem is a neat analogue of Theorem~\ref{t:kohsimple}, and turns out to have deep consequences to the theory of Bessel models for automorphic forms, as well as to non-vanishing of periods and simultaneous non-vanishing of dihedral twists of modular $L$-functions. These applications will form the heart of Section~\ref{s:app}. It is perhaps surprising (given how easy the proof of Theorem~\ref{t:kohsimple} is) that Problem~\ref{siegelsimp} does not appear to be amenable to any attacks relying  on the theory of Hecke operators or local representation theory (even in the simplest case $N=1$).

One can be much more ambitious, of course, and take $V$ as in Theorem~\ref{t:ibukat}. This leads to the following significant generalization of Problem~\ref{siegelsimp}.

\begin{problem}\label{siegelhard}Let $F \in S_k^{(2)}(N)$ be a non-zero element in the orthogonal complement of $ S_k^{(2), \text{O}}(N)$. Show that there exists a fundamental matrix $S$ such that $a(F,S) \neq 0$.
\end{problem}

\subsection{The main result for Siegel cusp forms}

 The next theorem, which is a mild variation of~\cite[Thm. 2]{sahaschmidt} gives an affirmative answer, in a strong quantitative form, to Problem~\ref{siegelhard}, whenever $N$ is squarefree and $k > 2$ is even.

\begin{theorem}\label{t:mainsiegel}Let $k > 2$ and $N$ be a squarefree integer. Moreover, if $N>1$, assume that $k$ is even. Let $0 \ne F \in S_k^{(2)}(N)$ belong to the orthogonal complement of $ S_k^{(2), \text{O}}(N)$. Let $\delta>0$ be an absolute constant as in Theorem~\ref{t:mainhalf}. Then, one has the lower bound
$$|\{0 < d < X , \ d \text{ squarefree }, \ a(F, S) \ne 0 \text{ for some }  S  \text{ with }  d = -\disc (S)   \}| \gg_{F, \delta} X^\delta. $$

In particular, there are infinitely many fundamental matrices $S$ such that $a(F,S) \neq 0$.

\end{theorem}

Note that if $N =1$, the conditions $k>2$ and $F$ belonging to the orthogonal complement of $ S_k^{(2), \text{O}}(N)$ are both automatic.

\begin{sketch} By Theorem~\ref{t:ibukat}, there exists a \emph{primitive} matrix $T'$ such that $a(F, T') \neq 0$.  Suppose that $T '= \mat{a}{b/2}{b/2}{c}$. It is a classical result (going back at least to Weber) that the primitive quadratic form $ax^2 + bxy + cy^2$ represents infinitely many primes. So, let $x_0$, $y_0$ be such that $ax_0^2 + bx_0y_0 + cy_0^2$ is an odd prime not dividing $N$. Since this implies $\gcd(x_0, y_0) = 1$, we can find integers $x_1$, $y_1$ such that $A= \mat{y_1}{y_0}{x_1}{x_0} \in \SL_2(\Z).$ Then $T = \T{A}T'A$ has the property that $a(F, T) \neq 0$ and $T$ is of the form $\mat{a_0}{b_0/2}{b_0/2}{p}$ where $p$ is an odd prime not dividing $N$.

Note that so far, we have only used the $\SL_2(\Z)$-invariance of the Fourier coefficients and the theorem of Ibukiyama and Katsurada. To complete the proof, we are going to cook up a classical half-integral weight form and appeal to Theorem~\ref{t:mainhalf}. For simplicity, assume henceforth that $k$ is even. For all integers $n, r$ with $4np > r^2$, let us denote $$c(n,r) = a\left(F, \mat{n}{r/2}{r/2}{p}\right).$$ Now, let $$h(\tau) = \sum_{m=1}^\infty  c(m) e(m \tau).$$
where $$c(m) = \sum_{\substack{0 \le \mu \le 2p-1 \\ \mu^2 \equiv -m \pmod{4p}}} c\left((m+\mu^2)/4p, \ \mu \right).$$
By Theorem 4.8 of~\cite{manickram}, we know that $h \in S_{k-\frac{1}{2}}(4pN)$.

It is easy to see that $h$ is not identically equal to 0. Indeed put $d_0 = 4 a_0 p - b_0^2$. Then $c(d_0)$ equals $a\left(F, \mat{a_0}{b_0/2}{b_0/2}{p}\right) + a\left(F, \mat{a_0+p-b_0}{\ p - b_0/2}{p - b_0/2}{\ p}\right)$, which is simply $2 a\left(F, \mat{a_0}{b_0/2}{b_0/2}{p}\right)$ by~\eqref{siegelinv} and hence non-zero.

Now, by Theorem~\ref{t:mainhalf}, it follows that $$|\{0 < d<X, \ d \text{ squarefree }, \ c(d) \ne 0\}|\gg_{h,\delta} X^\delta. $$ For any of these $d$, there exists a $\mu$ such that  $c\left(\frac{d+\mu^2}{4p},\mu \right) = a \left(F, \mat{\frac{d+\mu^2}{4p}}{\mu/2}{\mu/2}{p}\right)$ is not equal to zero. This completes the proof.

\qed

\end{sketch}

\begin{remark}The construction of a half-integral weight form $h$ from the Fourier coefficients of $F$ in the proof above is best understood as arising from the isomorphism between the space of Jacobi forms and the space of modular forms of half-integral weight. In the case $N=1$, this isomorphism was first investigated in great detail by Skoruppa in his thesis.
\end{remark}

\begin{remark} In the applications, we will only need Theorem~\ref{t:mainsiegel} when $F$ is a Hecke eigenform. However, even in that case, the half-integral form $h$ constructed above need not be a Hecke eigenform! That is why, even if we wished to prove Theorem~\ref{t:mainsiegel} only for eigenforms, we would still need to prove Theorem~\ref{t:mainhalf} in general.

\end{remark}

\subsection{A mild restatement}

In this subsection, we will restate Theorem~\ref{t:mainsiegel} in a form that will be useful for applications. Let $-d <0$ be a fundamental discriminant and put $K = \Q(\sqrt{-d})$. Let $\Cl_K$ denote the ideal class group of $K$. It is a fact going back to Gauss that the $\SL_2(\Z)-$equivalence classes of binary quadratic forms of discriminant $-d$ are in natural bijective correspondence with the elements of $\Cl_K$. In view of~\eqref{siegelinv} and the comments immediately after, it follows that for any $c \in \Cl_K$ the notation $a(F, c)$ makes sense.

In particular, for any $F\in S_k^{(2)}(N)$, any imaginary quadratic field $K$ (with discriminant equal to $-d$) and any character $\Lambda$ of the finite group $\Cl_K$, we can make the definition

\begin{equation}\label{rfdef}
R(F, K, \Lambda) = \sum_{c \in \Cl_K}a(F, c) \Lambda^{-1}(c).
\end{equation}

An immediate corollary of Theorem~\ref{t:mainsiegel} is the following proposition.

\begin{proposition}\label{proprest}Let $k > 2$ and $N$ be a squarefree integer. Moreover, if $N>1$, assume that $k$ is even. Let $ F \in S_k^{(2)}(N)$ be a non-zero form belonging to the orthogonal complement of $ S_k^{(2), \text{O}}(N)$. Then there are infinitely many pairs $(K, \Lambda)$ with $K$ an imaginary quadratic field and $\Lambda$ an ideal class character of $\Cl_K$, such that  $R(F, K, \Lambda) \neq 0$.
\end{proposition}

\begin{proof} Take $K = \Q(\sqrt{-d})$ to be any field for which there exists a  semi-integral matrix $S  \text{ with }  -d = \disc (S)$ and $a(F, S) \neq 0$. The proposition is now clear from Theorem~\ref{t:mainsiegel} by Fourier inversion on the finite group $\Cl_K$.
\qed

\end{proof}

\section{Applications}\label{s:app}

\subsection{Existence of nice Bessel models for Siegel cusp forms}

The goal of this subsection is to explain a certain application of Theorem~\ref{t:mainsiegel} that was the original motivation for the author's work in this direction.

When working with automorphic representations on some reductive group, it is useful to have a \emph{model} for the representation consisting of a space of nicely transforming functions on the group. For automorphic representations on $\GL_n$, such a model is provided by the space of adelic Whittaker functions. Indeed, every cuspidal automorphic representation on $\GL_n(\A)$ has a Whittaker model, that is to say, the representation space consisting of Whittaker functions on $\GL_n(\A)$ contains (with multiplicity one!) each cuspidal automorphic representation of $\GL_n(\A)$. For automorphic representations on other classical groups, a Whittaker model does not necessarily exist. In particular, Siegel cusp forms of degree 2 correspond to automorphic representations on $\GSp_4(\A)$ which \emph{never} have Whittaker models. A useful substitute for the  missing Whittaker model is the so-called Bessel model.

 In contrast to Whittaker models, which are essentially canonical, Bessel
models depend on some arithmetic data. In the case of cuspidal automorphic representations on $\GSp_4(\A)$, there are two ingredients that go into the data defining a Bessel model\footnote{Here we are suppressing a third ingredient, which is the choice of additive character.}. One of them is a two by two non-degenerate symmetric matrix $S \in M_2(\Q)$ such that $-d = - 4 \det(S)$ belongs to $\Q^\times - (\Q^\times)^2 $. The other ingredient is a Hecke character $\Lambda$ of $K^\times \bs \A_K^\times$ where $K = \Q(\sqrt{-d})$. For any such $S$ and
$\Lambda$, it makes sense to ask whether a cuspidal automorphic representation $\pi$ of $\GSp_4(\A)$ has a global Bessel model of type $(S, \Lambda)$. The interested reader should consult~\cite{nov, fur, prasadbighash} for further details.

We say  that the type $(S, \Lambda)$ is \emph{fundamental}
if each of the following conditions is satisfied:
\begin{enumerate}
  \item $S=\mat{a}{b/2}{b/2}{1},\quad a, b \in \Z,  \quad -d := b^2 - 4a <0 \text{ is a fundamental discriminant}.$

\item The Hecke character $\Lambda$ is unramified at all finite places of $K=\Q(\sqrt{-d})$ and trivial at infinity,
and hence is a character   on the ideal class group of $K$.
\end{enumerate}

The automorphic representations we are interested in come from  Siegel cusp forms of degree 2. Indeed, let $F \in S_k^{(2)}(N)$
be an eigenform for the local Hecke algebras at all places. Then $F$ gives rise to an irreducible cuspidal representation $\pi_F$ of $\GSp_4(\A)$; see \cite{NPS}. It can be shown that there exist infinitely many pairs $(S, \Lambda)$ such that $\pi_F$ has a global Bessel model of type $(S, \Lambda)$. What is highly desirable, however, is that at least one of these pairs be fundamental, i.e., $\pi_F$ have at least one global Bessel model of fundamental type. Indeed, working with a fundamental global Bessel model means that all the associated local Bessel data is unramified, which makes it much easier to compute local zeta integrals and deduce various properties of global $L$-functions.

In a pioneering work, Furusawa~\cite{fur} used Bessel models to prove an integral representation for the $\GSp_4 \times \GL_2$ $L$-function of the twist of an eigenform $F$ in $S_k^{(2)}(1)$ with an elliptic (classical) eigenform $g$ of the same weight. He used this to deduce special value results in the spirit of Deligne's conjecture. However, his results were all obtained under the assumption that $\pi_F$ has a global Bessel model of fundamental type. Subsequent works by various people built upon Furusawa's results and proved various analytic, special value and functoriality properties for $\pi_F$ and its various associated twisted $L$-functions. We refer the reader to the papers~\cite{pitsch, saha1, sah2, transfer} for details. Since all these works depended upon Furusawa's unramified calculation, they were also only valid for Siegel cusp forms which have a global Bessel model of fundamental type. And till recently, it was completely unknown how restrictive this assumption about existence of fundamental Bessel model is.

The link between global Bessel models  and the theme of this article is provided by the following lemma.

\begin{lemma}\label{lemmabesel}Let $N$ be squarefree and $F \in S_k^{(2)}(N)$ be an eigenform for the local Hecke algebras at all places. Let $\pi_F$ be the irreducible cuspidal representation attached to $F$. The following are equivalent:

\begin{enumerate}
\item $\pi_F$ has a global Bessel model of fundamental type $(S, \Lambda).$

\item $R(F, K, \Lambda) \neq 0$ where $K = \Q(\sqrt{\disc(S)})$ is the imaginary quadratic field associated to $S$ as above.
\end{enumerate}

\end{lemma}

In other words, $\pi_F$ has a  global Bessel model of fundamental type , if and only if $F$ has a non-zero fundamental Fourier coefficient!

\begin{proof} This follows from~\cite[Prop. 4.3]{sahaschmidt} using recent work of Pitale and Schmidt~\cite{pitschbessel} that tells us that a local eigenvector fixed by the Siegel parahoric is a test vector for the local Bessel functional.

\qed

\end{proof}

Together, Proposition~\ref{proprest} and Lemma~\ref{lemmabesel} immediately imply the following theorem.

\begin{theorem}Let $N$ be squarefree and $F \in S_k^{(2)}(N)$ lie in the orthogonal complement of $ S_k^{(2), \text{O}}(N)$ and be an eigenform for the local Hecke algebras at all places. Let $\pi_F$ be the irreducible cuspidal representation attached to $F$. Then $\pi_F$ has a global Bessel model of fundamental type.
\end{theorem}

Thus, our work on determination of Siegel cusp forms by fundamental Fourier coefficients has the very pleasant consequence of making \emph{unconditional} all the results mentioned earlier due to Furusawa, Pitale, Schmidt and the author. It is worthwhile to remark here that most of those results are for $N=1$; however, recent work by Pitale and Schmidt~\cite{pitschbessel} should allow them to be extended to the case of arbitrary squarefree $N$.

\subsection{Simultaneous non-vanishing of $L$-functions}\label{s:yosh}

In this subsection, we will describe an application of Proposition~\ref{proprest} to simultaneous non-vanishing of dihedral twists of modular $L$-functions. \emph{This work was done jointly with Ralf Schmidt and will appear in more detailed form in~\cite{sahaschmidt}.}

Let $\Lambda$ be an ideal class group character of an imaginary quadratic field $K$ of discriminant $-d$ and $f$ be a classical holomorphic newform.  One can then form the $L$-function $L(s, \pi_f \times \theta_\Lambda)$; this is the Rankin--Selberg convolution of the automorphic representation $\pi_f$ attached to $f$ and the $\theta$-series $$\theta_\Lambda(z) = \sum_{0 \ne \a \subset O_K}\Lambda(\a) e(N(\a)z).$$
Here, $\theta_\Lambda$ is a holomorphic modular form of weight $1$ and nebentypus $\left(\frac{-d}{*} \right)$ on $\Gamma_0(d)$; it is a cusp form if and only if $\Lambda^2 \ne 1$.

The problem of studying the non-vanishing of the central values $L(\frac12, \pi_f \times \theta_\Lambda)$ arises naturally in several contexts, and a considerable amount of work has been done in this direction. In this subsection, we will describe how our work so far leads to a \emph{simultaneous} non-vanishing result for $L(\frac12, \pi_f \times \theta_\Lambda)$, $L(\frac12, \pi_g \times \theta_\Lambda) $ for two fixed forms $f$, $g$ (but varying $K$ and $\Lambda$) under certain hypotheses.

But first, back to Siegel cusp forms. Let $F \in S_k^{(2)}(N)$ be a Hecke eigenform and $\pi_F$ be the irreducible cuspidal representation attached to $F$. Recall the definition of $R(F, K, \Lambda)$ from~\eqref{rfdef}. A generalization of a famous conjecture of B\"ocherer~\cite{boch-conj} by several people  (Furusawa, Shalika, Martin, Prasad, Takloo-Bighash), leads to the following very interesting Gross-Prasad type conjecture.

\begin{conjecture}\label{bochconj} Let $F \in S_k^{(2)}(N)$ be a Hecke eigenform, $K$ an imaginary quadratic field and $\Lambda$ an ideal class group character of $K$. Suppose that $R(F, K, \Lambda) \neq 0.$ Then $L(\frac12, \pi_F \times \theta_\Lambda) \neq 0$.
\end{conjecture}

While Conjecture~\ref{bochconj} is open at the moment, it has been proved for certain special Siegel cusp forms known as Yoshida lifts.

The data required to define a Yoshida lift are as follows. Let  $N_1, N_2$ be two squarefree integers that are \emph{not} coprime and put $N = \text{lcm}(N_1, N_2)$. Let $f$ be a classical newform of weight $2$ on $\Gamma_0(N_1)$ and $g$ be a classical newform of weight $2k$ on $\Gamma_0(N_2)$. Finally, assume that for each prime $p$ dividing $\text{gcd}(N_1, N_2)$, $f$ and $g$ have the same Atkin-Lehner eigenvalue at $p$.
\bigskip

\textbf{Existence of Yoshida lift:} \emph{Under the above assumptions, there exists a non-zero element $F \in S_{k+1}^{(2)}(N)$ that has the following properties}:
\begin{enumerate}

\item $F$ lies in the orthogonal complement of $ S_{k+1}^{(2), \text{O}}(N)$.

\item $F$ is an eigenform for the local Hecke algebras at all places.

\item For all automorphic representations $\sigma$ on $\GL_n(\A)$ for each $n$, we have $L(s, \pi_F \times \sigma) = L(s, \pi_f \times \sigma)L(s, \pi_g \times \sigma).$ Here and elsewhere, $L(s,  \ )$ denotes the complete global Langlands $L$-function for the relevant representations.

\end{enumerate}
\bigskip

The above lift was first investigated by Yoshida and has been studied extensively by B{\"o}cherer and Schulze-Pillot. In fact, the Yoshida lift is a certain case of Langlands functoriality; see Sect.~\cite[Sec. 3.2]{sahaschmidt} for more details.

We briefly explain how the Yoshida lift $F$ is constructed from the classical cusp forms $f$ and $g$. First, we fix a definite quaternion algebra $D$ which is unramified at all finite primes outside $\gcd(N_1, N_2)$. Via the Jacquet-Langlands correspondence, we transfer   $\pi_f$, $\pi_g$ to representations $\pi_f'$, $\pi_g'$ on $D^\times(\A)$. Using the isomorphism $$(D^\times \times D^\times) / \Q^\times \cong GSO(4)$$ we obtain an automorphic representation $\pi'_{f,g}$ on $GSO(4, \A)$. We use the theta lifting to transfer $\pi'_{f,g}$ to the automorphic representation $\pi_F$ on $\GSp_4(\A)$. Finally, the form $F$ is constructed by picking a suitable vector in the space of $\pi_F$. For the details of this last step, as well as a more thorough explanation of Yoshida lifts from the representation theoretic point of view, we refer the reader to the author's paper with Schmidt~\cite[Sec. 3]{sahaschmidt}.

Now, recall the Conjecture~\ref{bochconj} stated earlier. It turns out that this conjecture has been recently proved for Yoshida lifts by Prasad and Takloo-Bighash.

\begin{theorem}[Prasad--Takloo-Bighash~\cite{prasadbighash}] \label{prasadbigth}
Let $f, g$ be classical newforms as above, and let $F \in S_{k+1}^{(2)}(N)$ be their Yoshida lift. Suppose that $K$ is an imaginary quadratic field and $\Lambda$ an ideal class group character of $K$  such that $R(F, K, \Lambda) \neq 0.$ Then $L(\frac12, \pi_F \times \theta_\Lambda) \neq 0$.
\end{theorem}

Putting together Proposition~\ref{proprest}, Theorem~\ref{prasadbigth} and the defining property of the Yoshida lift, we deduce the following result on simultaneous non-vanishing of $L$-functions.

\begin{theorem}[\cite{sahaschmidt}]\label{th:simulnonvan}Let $k>1$ be an odd integer. Let $N_1$, $N_2$ be two positive, squarefree integers such that $M = \gcd(N_1, N_2)>1$. Let $f$ be a holomorphic newform of weight $2$ on $\Gamma_0(N_1)$ and $g$ be a holomorphic newform of weight $2k$ on $\Gamma_0(N_2)$. Assume that for all primes $p$ dividing $M$, the Atkin-Lehner eigenvalues of $f$ and $g$ coincide. Then there exists an imaginary quadratic field $K$ and a character $\Lambda \in  \widehat{\Cl_K}$ such that $L(\frac12, \pi_f \times \theta_\Lambda) \ne 0$ and $L(\frac12, \pi_g \times \theta_\Lambda) \ne 0$. In fact, if $D(f,g)$ is the set of $d$ satisfying the following conditions:

 \begin{enumerate}
 \item $d>0$ is an odd, squarefree integer and $-d$ is a fundamental discriminant,
  \item \label{part 2} There exists an ideal class group character $\Lambda$ of $K=\Q(\sqrt{-d})$ such that  $L(\frac12, \pi_f \times \theta_\Lambda) \neq 0$ and $L(\frac12, \pi_g \times \theta_\Lambda) \neq 0$,
 \end{enumerate}

then, for an appropriate absolute constant $\delta>0$, one has the lower bound
\begin{equation} \label{bounddelta}|\{0<d<X, \ d\in  D(f,g) \}| \gg_{f,g,\delta} X^\delta.\end{equation}

\end{theorem}

For related work on non-vanishing of $L$-functions, we refer the reader to the introductions of the papers \cite{micven07} and~\cite{ono-skinner}.

\begin{remark}By considering Yoshida lifts for more general (i.e. non squarefree) levels and more general congruence subgroups, it is possible to relax some of the conditions on $f$, $g$, $N_1$, $N_2$ above. For more details, we refer the reader to~\cite{sahapetersson}.

\end{remark}

\begin{remark}
We say a few words about the restrictions on $f$ and $g$ in Theorem~\ref{th:simulnonvan}. The conditions that $N_1, N_2$ are squarefree and that the Atkin-Lehner eigenvalues of $f$ and $g$ coincide are needed to ensure that there exists a holomorphic Yoshida lift attached to $(f,g)$ with respect to a Siegel-type congruence subgroup $\Gamma_0^{(2)}(N) \subset \Sp_4(\Z)$ of squarefree level $N$. Indeed, our key result (Theorem~\ref{t:mainsiegel}) on non-vanishing fundamental Fourier coefficients is only proved for Siegel cusp forms with respect to such congruence subgroups. However, even if these two conditions are removed, $(f,g)$ will still have a Yoshida lift (possibly with respect to some other congruence subgroup) provided that there is a prime $p$ dividing $\gcd(N_1, N_2)$ such that $\pi_f$, $\pi_g$ are both discrete series at $p$.\footnote{The restriction that there is a prime dividing $\gcd(N_1, N_2)$ where $\pi_f$, $\pi_g$ are both discrete series will probably be very difficult to remove by our method, because without this condition there are no Jacquet-Langlands transfers and hence no (holomorphic) Yoshida lifts. It is conceivable that one could still consider the ``Fourier coefficients" of the non-holomorphic Yoshida lift in this setup and prove a non-vanishing result for those.} So, an analogue of Theorem~\ref{t:mainsiegel} for Siegel cusp forms with respect to more general congruence subgroups will allow us to remove some of the restrictions on $f$ and $g$. This is currently work in progress by J. Marzec at the University of Bristol. To remove the restriction on the weight of $g$ would require us to extend Theorem~\ref{t:mainsiegel} to vector valued Siegel cusp forms, which seems possible at least in principle.
\end{remark}

\begin{acknowledgement}I would like to thank Emmanuel Kowalski, Ameya Pitale and Ralf Schmidt for helpful comments on an earlier draft of this article. I would also like to thank Ramin Takloo-Bighash for suggesting the application of Section~\ref{s:yosh} to me.
\end{acknowledgement}

\myaddress{ETH Z\"urich -- D-MATH, R\"amistrasse 101, 8092 Z\"urich, Switzerland} \myemail{abhishek.saha@math.ethz.ch}

  \end{document}